\documentclass{article}

\usepackage[pdftex]{graphicx}
\usepackage{amsmath}
\usepackage{amssymb}
\usepackage{amsthm}
\usepackage{amsfonts}
\usepackage{wrapfig}
\usepackage{calc}
\usepackage{tabulary}
\usepackage{fancyhdr}
\usepackage[plain]{fullpage}
\usepackage{xcolor}
\usepackage{tikz-cd}
\usepackage{enumerate}
\usepackage{accents}
\usepackage{lineno}

\newcommand{\aut}{\textnormal{Aut}}
\newcommand{\emb}{\textnormal{Emb}}

\newcommand{\mon}{\textnormal{End}}
\newcommand{\pol}{\textnormal{Pol}}

\newcommand{\supp}{\textnormal{supp}}

\newcommand{\interior}{\accentset{\circ}}
\newcommand{\varBA}{\mathrm{var}}

\makeatletter
\providecommand*{\cupdot}{%
  \mathbin{%
    \mathpalette\@cupdot{}%
  }%
}
\newcommand*{\@cupdot}[2]{%
  \ooalign{%
    $\m@th#1\cup$\cr
    \sbox0{$#1\cup$}%
    \dimen@=\ht0 %
    \sbox0{$\m@th#1\cdot$}%
    \advance\dimen@ by -\ht0 %
    \dimen@=.5\dimen@
    \hidewidth\raise\dimen@\box0\hidewidth
  }%
}

\providecommand*{\bigcupdot}{%
  \mathop{%
    \vphantom{\bigcup}%
    \mathpalette\@bigcupdot{}%
  }%
}
\newcommand*{\@bigcupdot}[2]{%
  \ooalign{%
    $\m@th#1\bigcup$\cr
    \sbox0{$#1\bigcup$}%
    \dimen@=\ht0 %
    \advance\dimen@ by -\dp0 %
    \sbox0{\scalebox{2}{$\m@th#1\cdot$}}%
    \advance\dimen@ by -\ht0 %
    \dimen@=.5\dimen@
    \hidewidth\raise\dimen@\box0\hidewidth
  }%
}
\makeatother

\makeatletter
\providecommand{\bigsqcap}{%
  \mathop{%
    \mathpalette\@updown\bigsqcup
  }%
}
\newcommand*{\@updown}[2]{%
  \rotatebox[origin=c]{180}{$\m@th#1#2$}%
}
\makeatother

\newtheorem{lemma}{Lemma}[section]
\newtheorem{theorem}[lemma]{Theorem}
\newtheorem{cor}[lemma]{Corollary}

\newtheorem{dfn}[lemma]{Definition}
\newtheorem{prop}[lemma]{Proposition}

\newtheorem{remark}[lemma]{Remark}

\begin{document}

\title{Automatic Homeomorphicity of Locally Moving Clones}
\author{Robert Barham\footnote{The author has received funding from the European Research Council under the European Community's Seventh Framework Programme (FP7/2007-2013 Grant Agreement no. 257039).} \\ Institut f\"ur Algebra, TU Dresden \\ robert.barham@yahoo.co.uk}

\maketitle

\begin{abstract}
We extend the work of M. Rubin on locally moving groups to clones, showing that a locally moving polymorphism clone has automatic homeomorphicity with respect to the class of all polymorphism clones.  We show that if $\pol(M,\mathcal{L})$ is a reduct of $(\mathbb{Q},<)$ or $(\mathbb{L},C)$ such that:
\begin{enumerate}[i.]
\item $\aut(M,\mathcal{L}) \not= \aut(M,=)$; and
\item $\mon(M,\mathcal{L}) = \emb(M,\mathcal{L})$,
\end{enumerate}
then $\pol(M,\mathcal{L})$ is locally moving (and hence has automatic homeomorphicity with respect to the class of all polymorphism clones), where $\mathbb{Q}$ is the rationals, $(\mathbb{L},C)$ is the infinite binary-branching homogeneous $C$-relation.  We also show that if $\mathcal{M}=(\mathbb{B}, \cup  , \cap, \,^c, 1, 0)$, the Fra\"{i}ss\`{e} Generic Boolean algebra, then $\pol(\mathcal{M})$ is locally moving.
\end{abstract}

\section{Introduction}

\paragraph{}
Automorphism groups, endomorphism monoids and polymorphism clones come with a natural topology attached, the topology of pointwise convergence (see 3.2 of \cite{BPPRecon}).  Recovering the topological structure of a polymorphism clone from its abstract structure is a problem that has recently attracted a lot of interest.  One reason for this is its applicability to theoretical computer science, specifically to constraint satisfaction problems (CSPs) with countably categorical templates.  These are computational problems which ask whether a given finite structure (the `input') homomorphically embeds into another, possibly infinite, structure (the `template').  If two countably categorical templates have homeomorphic polymorphism clones then their constraint satisfaction problems are polynomial-time equivalent, \cite{BodirskyPinskerTopBirkhoff}.  Therefore, if a polymorphism clone has `automatic homeomorphicity', i.e. every abstract isomorphism is already a homeomorphism, then the complexity of the corresponding CSP is determined by the equational theory of that clone.

The first results in this direction appears in a paper by Bodirsky, Pinsker and Pongr\'{a}cz, \cite{BPPRecon}.  In this paper, known results about the reconstruction of the topological structure of automorphism groups are extended to polymorphism clones using `gate coverings'.  One application of their method is to show that $\pol(V,E)$, the polymorphism clone of the random graph with the edge relation, has automatic homeomorphicity with respect to all closed subclones of $\pol(V,=)$.

Their paper, however, was unable to deal with the polymorphism clone of the rational numbers as a dense linear order, which is a particularly important template.  Pech and Pech in \cite{PechPechMonoid} extended the method in Bodirsky, Pinsker and Pongr\'{a}cz using `universal homogeneous endomorphisms' to show that $\mon(\mathbb{Q},\leq)$ has automatic homeomorphicity with respect to the polymorphism clones of countable partial orders, amongst other results concerning the rational Urysohn sphere and the generic partial order.  Behrisch, Truss and Vargas-Garc\'{i}a are preparing a paper \cite{BTVGQ} that shows that $\pol(Q,\leq)$ has automatic homeomorphicity with respect to all closed subclones of $\pol(Q,=)$.  Their result extends the method of Bodirsky, Pinsker and Pongr\'{a}cz using endomorphisms with a `maximally spread out image'.  The special endomorphisms of Pech and Pech and the special endomorphims of Behrisch, Truss and Vargas-Garc\'{i}a appear to be strongly related, and make an appearance here as `conjugable' endomorphisms.

In this paper, rather than adapt the method of Bodirsky, Pinsker and Pongr\'{a}cz to work in other circumstances, I exploit the `locally moving' method of Rubin for studying reconstruction from automorphism groups, which is discussed shortly.  This has three advantages, firstly the method of Bodirsky, Pinsker and Pongr\'{a}cz needs to already know that the automorphism group has automatic homeomorphicity, but here the method for showing that is included, so it can potentially be applied to structures where this was not previously known.  The second is that it establishes automatic homeomorphicity with respect to all polymorphism clones, not just polymorphism clones of countable structures.  The third, and most important, is that this can simultaneously produce reconstruction proofs for entire classes of clones that are important to the theory of CSPs.

The first section of this paper develops the notion of a locally moving clone, and shows that all locally moving clones have automatic homeomorphicity with respect to all polymorphism clones.  A clone must satisfy two conditions if it is to be locally moving:
\begin{enumerate}
\item the group of invertible elements must be a locally moving group; and
\item for every a unary function $f$ there is an `algebraically canonical' $g$ such that $gf$ is algebraically canonical'.  A full definition will be given shortly, where its close resemblance to conjugation will become apparent.  A consequence of this assumption is that every unary function in the clone must be injective.
\end{enumerate}

The first condition comes directly from the work of M. Rubin on `locally moving groups', a notion he created to study the reconstruction problem for the automorphism groups of Boolean algebras \cite{RubinHBA}.  He subsequently used this approach to study the reconstruction problem for topological spaces \cite{Rubin1989}, linear orders (with S. McCleary) \cite{RubinOrder} and trees \cite{Rubin91}.

His results on locally moving groups are essential to this paper, so I recommend to the reader \cite{RubinLMG}.  The results of Rubin's that I use are given without proof in this paper, but they can be found with proof in Section 2 of \cite{RubinLMG}.  When I have quoted his results, I have also given the theorem numbers from that work.

Rubin's work on groups makes heavy use of conjugation, so obviously we need something to replace that if we are to work with clones.  The second condition allow the use of an ersatz conjugation in the monoid.  Interestingly, no assumptions about the clone proper are needed, just the unary part.  This means that if we show that a particular polymorphism clone is locally moving then we have shown that every polymorphism clone with the same endomorphism monoid is locally moving.  This lets us prove that infinitely many clones have automatic homeomorphicity with a single proof.

After some general methods are established, we show that if $\pol(M,\mathcal{L})$ is a reduct of $(\mathbb{Q},<)$ or $(\mathbb{L},C)$ such that:
\begin{enumerate}[i.]
\item $\aut(M,\mathcal{L}) \not= \aut(M,=)$; and
\item $\mon(M,\mathcal{L}) = \emb(M,\mathcal{L})$,
\end{enumerate}
then $\pol(M,\mathcal{L})$ is locally moving.  In this list, $(\mathbb{Q},<)$ is the rationals with a linear order, $(\mathbb{L},C)$ is the infinite binary-branching homogeneous $C$-relation.  We also examine the Fra\"{i}ss\'{e} generic Boolean algebra.

I am indebted to Truss and Vargas-Garcia for sharing with me their work on automatic homeomorphicity for $\mon(\mathbb{Q},<)$, $\mon(\mathbb{Q},\leq)$, and $\pol(\mathbb{Q},\leq)$.  Their notion of endomorphisms whose image is as spread out as possible was very useful.  I am also grateful to Trung Van Pham, for producing extremely illuminating counter-examples to earlier drafts, and Manuel Bodirsky, for his advice, questions and comments.

\section{Locally Moving Clones}\label{section:locallymoving}

\subsection{Groups}

Let $\mathcal{L}_{Gp} = \lbrace \circ , \mathrm{id}, \;^{-1} \rbrace$ be the language of groups.

\begin{dfn}  Let $B$ be a complete atomless Boolean algebra, let $x \in B$ and let $\alpha \in \aut(B)$.
\begin{enumerate}
\item $\mathrm{Fix}(\alpha):=\lbrace y \in B \: : \: \forall x \subseteq y \: \alpha(x) = x \rbrace$.
\item $\mathrm{fix}(\alpha):= \bigcup\; \mathrm{Fix}(\alpha)$.
\item $\varBA(\alpha) :=  -\mathrm{fix}(\alpha)$.
\item If $G \leq \aut(B)$ then $\mathrm{Var}(B,G) := \lbrace \varBA(\beta) \: : \: \beta \in G \rbrace$.
\end{enumerate}
\end{dfn}

\begin{dfn}
A group $G$ is said to be \emph{locally moving} if there is a complete atomless Boolean algebra $B$ such that $G \leq \aut(B)$ and for all $b \in B$ there is an $\alpha \in G$ such that $\varBA(\alpha) \leq b$.
\end{dfn}

\begin{lemma}[(Rubin, Proposition 2.3 of \cite{RubinLMG})]
$\varBA(\alpha)= \bigcup \lbrace a \in B \: : \: a \cap \alpha(a) = 0 \rbrace$
\end{lemma}

\begin{theorem}[(Rubin's Expressibility Theorem, Theorem 2.5 of \cite{RubinLMG})]
Let $G$ be any group acting on any Boolean algebra $B$.  There are $\mathcal{L}_{Gp}$-formulas $\phi_{Eq}$ $\phi_{\leq}$ and $\phi_{Ap}$ such that
\[
\begin{array}{r c l}
G \models \phi_{Eq}(\alpha,\beta) & \Leftrightarrow & \varBA(\alpha) = \varBA(\beta) \\
G \models \phi_{\subseteq}(\alpha,\beta) & \Leftrightarrow & \varBA(\alpha) \subseteq \varBA(\beta) \\
G \models \phi_{Ap}(\alpha,\beta,\gamma) & \Leftrightarrow & \alpha(\varBA(\beta)) = \varBA(\gamma) \\
\end{array}
\]
\end{theorem}

\begin{theorem}[(Rubin, Corollary 2.12 of \cite{RubinLMG})]
There is a sentence $\sigma_{LM}$ in $\mathcal{L}_{Gp}$ such that $G \models \sigma_{LM}$ if and only if $G$ is a locally moving group.
\end{theorem}

\begin{theorem}[(Rubin's Reconstruction Theorem for Local Movement Systems, Corollary 2.10 of \cite{RubinLMG})]\label{RubinInduced}
Let $G$ and $H$ be locally moving.  Let $B$ and $C$ be Boolean algebras that witness this, and let $\theta: G \rightarrow H$ be an isomorphism.  Then there is an isomorphism $\tau: B \rightarrow C$ such that
\[\theta(\alpha) = \tau \alpha \tau^{-1} \]
\end{theorem}

\begin{remark}\label{remark:neededdetail}
$\alpha(\varBA(\beta))=\varBA(\beta^\alpha)$.  Rubin exploits this fact to construct $\phi_{Ap}$.  However we cannot conjugate by non-invertible elements, so it is not straightforward to extend this to clones.
\end{remark}

\subsection{Clones}

Let $\Delta$ be a clone.  Let $\Delta_n$ be the functions of arity $n$.  We denote the language of clones by $\mathcal{L}_{Cl}= \lbrace \circ_{m,n}, \pi^n_m \: : \: n,m \in \mathbb{N} \rbrace$, where $\circ_{n,m}: \Delta_n \times (\Delta_m)^n \rightarrow \Delta_m $ is composition, and $\pi^n_i$ is the projection from an $n$-tuple to the $i^{\mathrm{th}}$-coordinate.

\begin{dfn}
$\phi_{\mathrm{Aut}}(f):=\exists g \: (fg=\mathrm{id}=gf)$.  Note that $\phi_{\mathrm{Aut}}(\Delta)$ is the group of invertibles.
\end{dfn}

The full group structure is definable on $\phi_{\mathrm{Aut}}(\Delta)$ using $\mathcal{L}_{Cl}$.  If we add the requirement that every variable satisfies $\phi_\mathrm{Aut}$ to $\phi_{Eq}$, $\phi_{\subseteq}$ and $\phi_{Ap}$ then Rubin's Expressibility Theorem still holds for clones.  From now on, suppose that $\phi_\aut(\Delta)$ is a locally moving group.

\begin{dfn}
Let $\zeta $ be the following formula for $f \in \Delta_1$:
\[\zeta(f):= \forall \alpha \in \phi_\aut(\Delta) \: \exists \beta,\gamma \in \phi_\aut(\Delta) \: ( \phi_{Eq}(\alpha,\gamma) \wedge f \gamma = \beta f )\]
If $\Delta \models \zeta(f)$ then we call $f$ \emph{algebraicially canonical}.
\end{dfn}

\begin{lemma}\label{Lemma:IHopeThisWorks}
Suppose $f \in \Delta_1$.  Let $\alpha, \beta \in \phi_{\mathrm{Aut}}(\Delta)$ be such that $f\alpha = \beta f$.  Then $\varBA(\beta)=f(\varBA(\alpha))$.
\end{lemma}
\begin{proof}
Suppose $a \in \mathrm{Fix}(\alpha)$, i.e. if $b \subseteq a$ then $g(b)=b$.  Then $\beta(f(b))=f\alpha(b)=f(b)$ and so $f(a) \in \mathrm{Fix}(\beta)$.  Therefore $f (\mathrm{fix}(\alpha)) \subseteq \mathrm{fix}(\beta)$ and thus $\varBA(\beta)\subseteq f(\varBA(\alpha))$.

Let $a \in \varBA(\alpha)$, i.e. $\alpha(a) \cap a = 0$.  Then $f\alpha(a) \cap f(a) = 0$, and therefore $\beta f(a) \cap f(a) = 0$, and thus $f(a) \in \varBA(\beta)$.  This shows that
\[f(\lbrace a \in B \: : \:  a \cap \alpha(a)=0 \rbrace) \subseteq \lbrace a \in B \: : \: a \cap \beta(a) =0 \rbrace \]
and therefore $f(\varBA(\alpha)) \subseteq \varBA(\beta)$.
\end{proof}

We now have a substitute for conjugation, but we can only apply it to the elements that realise $\zeta$.  Some more work, and an assumption, is needed to extend this to all unary functions.

\begin{dfn}
\[\phi^{\zeta}_{AP}(f,\alpha,\beta):= \zeta(f) \wedge \phi_{\mathrm{Aut}}(\alpha) \wedge \phi_{\mathrm{Aut}}(\beta) \wedge \exists \gamma, \delta \in \phi_\aut(\Delta) \: \left(
\begin{array}{c}
  \phi_{Eq}(\beta,\gamma)\wedge  \phi_{Eq}(\alpha,\delta)\\
  \wedge f \delta =\gamma f 
\end{array}
\right)
\]

An English translation of $\phi^\zeta_{Ap}$ is as follows: `$f$ satisfies $\zeta$ and $\alpha$ and $\beta$ are automorphisms.  There are automorphisms $\gamma, \delta$ such that $f \delta = \gamma f$, $\varBA(\beta)=\varBA(\gamma)$ and $\varBA(\alpha)=\varBA(\delta)$.'
\end{dfn}

\begin{lemma}
$\Delta \models \phi^\zeta_{AP}(f,\alpha,\beta)$ if and only if $\Delta \models \zeta(f)$ and $f(\varBA(\alpha))=\varBA(\beta)$.
\end{lemma}
\begin{proof}
First suppose that $f(\varBA(\alpha))=\varBA(\beta)$.  Since $\Delta \models \zeta(f)$ we can find automorphisms $\gamma, \delta $ such that $f \delta = \gamma f$ and $\varBA(\alpha)=\varBA(\delta)$.  Lemma \ref{Lemma:IHopeThisWorks} shows that $\varBA(\gamma)=f(\varBA(\delta))=f(\varBA(\alpha))=\varBA(\beta)$, and therefore $\Delta \models \phi^{\zeta}_{Ap}(f,\alpha,\beta)$.

Now suppose $\Delta \models \phi^{\zeta}_{Ap}(f,\alpha,\beta)$.  Lemma \ref{Lemma:IHopeThisWorks} shows that $\varBA(\gamma)=f(\varBA(\delta))$, so $f(\varBA(\alpha))=\varBA(\beta)$. 
\end{proof}

\begin{dfn}
Let $\zeta_{inj}$ be the following formula for $f \in \Delta_1$:
\[\zeta_{inj}(f):= \zeta(f) \wedge \forall \alpha, \beta, \gamma \in \phi_\aut(\Delta) ((\phi^\zeta_{AP}(f, \alpha, \gamma) \wedge \phi^\zeta_{AP}(f, \beta, \gamma) ) \rightarrow \phi_{Eq}(\alpha,\beta))\]
\end{dfn}

\begin{lemma}\label{Lemma:Injective}
If $\Delta \models \zeta_{inj}(f)$ then $f$ is injective on $\lbrace \varBA(\alpha) \: : \: \alpha \in \phi_\aut(\Delta) \rbrace$.
\end{lemma}
\begin{proof}
The formula $\zeta_{inj}$ says that if $f(\varBA(\alpha))=f(\varBA(\beta)$ then $\varBA(\alpha)=\varBA(\beta)$.
\end{proof}

\begin{dfn}\label{dfn:localmoveclone}
$\Delta$ is a \emph{locally moving clone} if it satisfies the following two conditions:
\begin{description}
\item[The Group Condition] $\phi_{\mathrm{Aut}}(\Delta)$ is a locally moving group, acting on complete atomless Boolean algebra $B$.

\item[The Monoid Condition] \[\Delta \models \forall f  , \alpha \in \Delta_1 \: \left( \phi_\aut(\alpha)  \rightarrow \exists g, \beta,\gamma \in \Delta_1\: \left(
\begin{array}{c}

\phi_\aut(\beta) \wedge \phi_\aut(\gamma) \wedge \\
\zeta_{inj}(g)  \wedge  \phi_{Eq}(\alpha,\gamma) \\
 \wedge ( g f) \gamma = \beta ( g f )
\end{array}
\right)\right)\]

In English, `For all $f \in \Delta_1$ and all $\alpha \in \phi_\aut(\Delta)$ there are $g \in \zeta_{inj}(\Delta)$ and $\beta,\gamma \in \phi_\aut(\Delta)$ such that $\varBA(\alpha)=\varBA(\gamma)$ and $gf \gamma=\beta fg$.'
\end{description}
\end{dfn}

The Monoid Condition as written is equivalent to the far simpler `$\forall f \in \Delta_1 \exists g \in \Delta_1 ( \zeta(g) \wedge \zeta(gf))$', however using the more complicated version will simplify subsequent proofs, as well as make the notation used in those proofs easier to follow.

\begin{lemma}\label{Remark:FODef}
There is a sentence $\sigma_{LM}$ such that $\Delta \models \sigma_{LM}$ if and only if $\Delta$ is a locally moving clone.
\end{lemma}
\begin{proof}
 Condition 1 of \ref{dfn:localmoveclone} is expressible in a first order way, because the group part of $\Delta$ and the its structure are definable, and being a locally moving group is expressible.  Condition 2 gives the rest of the sentence explicitly.

\end{proof}

We are now ready to get an analogue of Rubin's Expressibility Theorem for clones.

\begin{theorem}\label{Theorem:DefAction}
Let $\Delta$ be a locally moving clone, acting on complete atomless Boolean algebra $B$, and let $n$ be any natural number.  Let $f \in \Delta_n$, let $\bar{\alpha} \in \phi_{\mathrm{Aut}}(\Delta)^n$, and let $\beta \in \phi_{\mathrm{Aut}}(\Delta)$.  There is a formula $\phi^{n}_{AP}(f,\bar{\alpha},\beta)$ such that $\Delta \models \phi^{n}_{AP}(f,\bar{\alpha},\beta)$ if and only if $f(\varBA(\bar{\alpha}))=\varBA(\beta) $.
\end{theorem}
\begin{proof}
We will first deal with $n=1$.  We define $\phi^{1}_{AP}(f,\alpha,\beta)$ to be the following formula:

\[ \phi_\aut(\alpha) \wedge \phi_\aut(\beta) \wedge \exists \gamma,\delta,g \in \Delta_1 \: \left(
\begin{array}{ c}
\zeta_{inj}(g) \wedge \phi_\aut(\gamma) \wedge \phi_\aut(\delta) \wedge \\
\phi_{Eq}(\alpha,\delta) \wedge ( gf \delta = \gamma gf )   \wedge \phi^\zeta_{Ap}(g,\beta,\gamma)
\end{array}
 \right)
\]

An English translation of $\phi^{1}_{Ap}$ is as follows: `$\alpha$ and $\beta$ are automorphisms.  There are $\gamma, \delta \in \phi_\aut(\Delta)$ and $g \in \zeta_{inj}(\Delta)$ such that $fg \delta = \gamma fg$, $\varBA(\alpha)=\varBA(\delta)$ and $g(\varBA(\beta))=\varBA(\gamma)$.'

First suppose that $f(\varBA(\alpha))=\varBA(\beta)$.  The Monoid Condition shows that we can find $g$ such that $\zeta_{inj}(g)$ and  automorphisms $\gamma$, $\delta$ such that $\varBA(\alpha)=\varBA(\delta)$ and $gf\delta = \gamma gf$.  Lemma \ref{Lemma:IHopeThisWorks} shows that $gf(\varBA(\delta))=\varBA(\gamma)$, so $g(\varBA(\beta))=\varBA(\gamma)$ implies that $f(\alpha)=\beta$.  Now suppose $\Delta \models \phi^{1}_{Ap}(f,\alpha,\beta)$.  Lemma \ref{Lemma:IHopeThisWorks} shows that $\varBA(\gamma)=gf(\varBA(\alpha))$, so $f(\varBA(\alpha)) \in g^{-1}(\varBA(\gamma))$.  By Lemma \ref{Lemma:Injective} $g^{-1}(\varBA(\gamma))=\varBA(\beta)$, so $f(\varBA(\alpha))=\varBA(\beta)$.

We now have a formula $\phi^{1}_{AP}$ such that $\Delta \models \phi^{1}_{AP}(f,\alpha,\beta) \Leftrightarrow f(\varBA(\alpha))=\varBA(\beta) $ which works for all $f \in \Delta_1$.  Using this formula, we define $\phi^{n}_{AP}$ to be the following formula:

\[
\left(
\begin{array}{l} \left( \bigwedge_{i=1}^{n}

\phi_{\mathrm{Aut}}(\alpha_i) \right) \wedge   \phi_{\mathrm{Aut}}(\beta)

\wedge \\

\forall \bar{l}\in \Delta_1 \bigwedge_{i=1}^{n} \left(  \forall \gamma \in \phi_\aut(\Delta) \left(
\begin{array}{c}
l(\gamma)=f(\alpha_1, \ldots , \alpha_{i-1}, \gamma, \alpha_{i+1} , \ldots \alpha_n)\\
\rightarrow \phi^{1}_{Ap}(l,\alpha_i,\beta)
\end{array}\right)\right)\end{array}\right)
\]

This formula insists that for every $i$, if $l \in \Delta_1$ is the function we get by fixing for all $j \not= i$ the $j^{th}$ argument of $f$ to be $\alpha_j$ then $l(\varBA(\alpha_i))=\varBA(\beta)$, thus guaranteeing that $f(\varBA(\bar{\alpha}))=\varBA(\beta) $.
\end{proof}

\begin{cor}
The action of $\Delta$ on $B(\Delta)$ is faithful.
\end{cor}

\begin{dfn}
Let $\Delta$ be a locally moving clone.  $B(\Delta)$ is the complete atomless Boolean algebra that $\Delta$ acts on.
\end{dfn}

\begin{prop}\label{thm:cloneinduced}
Suppose $\Delta$ is a locally moving clone and $\theta : \Delta \rightarrow \Gamma$ is an isomorphism.  Then there is an isomorphism $\tau: B(\Delta) \rightarrow B(\Gamma)$ such that $ \theta(f) = \tau f \tau^{-1} $ for all $f \in \Delta$.
\end{prop}
\begin{proof}
$\mathrm{Var}(\Delta)$ and $\mathrm{Var}(\Gamma)$ are definable in $\Delta$ and $\Gamma$ respectively, so $\theta$ induces an isomorphism from $\mathrm{Var}(\Delta)$ to $\mathrm{Var}(\Gamma)$.  Since these are both dense in $B(\Delta)$ and $B(\Gamma)$ respectively, this induced isomorphism extends to an isomorphism $\tau: B(\Delta) \rightarrow B(\Gamma)$.

The actions of $\Delta$ and $\Gamma$ are also definable, so this $\tau$ has the desired property.
\end{proof}

We've extended Rubin's work to clones, but currently we know that isomorphisms between locally moving clones are homeomorphisms of the topology with respect to the topologies arising from their actions on the Boolean algebras.  We are interested in the topology on $\pol(M)$ arising from the action on $M$, not $B(\pol(M))$.

\begin{theorem}\label{theorem:LMtopisRight}
If $\pol(M)$ is a locally moving clone then the topology on $\pol(M)$ from its action on $M$ and the topology from its action on $B(\pol(M))$ are the same.
\end{theorem}
\begin{proof}
Let $N := \lbrace F \subseteq B(\pol(M)) \: : \: F \textnormal{ is an ultrafilter} \rbrace$.  Let $f \in \pol(M)$ and let $x_1, \ldots ,x_n \in N$.  We define $f_N(x_1, \ldots x_n)$ to be the principal  ultrafilter of $f(x_1, \ldots x_n)$.

We define the following function from $B(\pol(M))$ to $\mathcal{P}(M)$.
\[\theta(\varBA(\alpha)) := \lbrace x \in M \: : \:\forall \gamma \in \aut(M) \: \gamma(x) \not= x \Rightarrow  \varBA(\gamma) \cap \varBA(\alpha) \neq \emptyset \rbrace\]

Let $M^+$ be the set of filters (not necessarily ultra-) of $\mathcal{P}(M)$.  We can embed $M$ into $M^+$ by mapping each $x$ to $\lbrace X \subseteq M \: : \: x \in X \rbrace$.  Since every polymorphism of $M$ extends uniquely to $\mathcal{P}(M)$, every polymorphism has a unique action on $M^+$ determined by its action on $M$.

Let $\tau_N$ be the topology on $\pol(M)$ from its action (via the $f_N$) on $N$.  Since $N$ is the space of ultrafilters of $B(\pol(M))$, the topology on $\pol(M)$ from its action on $B(\pol(M))$ is equal to $\tau_N$.  Let $\tau_M^+$ be the topology on $\pol(M)$ from its action on $M^+$.  For any $p \in \pol(M)$ the action of $p$ on $M$ determines the action of $p$ on $M^+ \setminus M$, so the topology on $\pol(M)$ from the action on $M$ is equal to $\tau_M^+$.

Let $X \subseteq (M^+)^n$ and let $f \in \pol(M) $.  Let $U:= \lbrace g \in \pol(M) \: : \: g|_X=f|_X \rbrace$ be a basic open set of $\tau_M^+$.  For each $x \in X$ let $a_x \in N$ be the ultrafilter that contains $\varBA(\alpha)$ if and only if $ \theta(\varBA(\alpha)) \in x$.  Let $A(X):= \lbrace a_x \: : \: x \in X \rbrace$.  Then $U = \lbrace g \in \pol(M) \: : \: g|_{A(X)} = f|_{A(X)} \rbrace$ and every basic open set of $\tau_M$ is a basic open of $\tau_N$.  Let $A \subseteq N^n$ and let $f \in \pol(M)$.  Let $U:= \lbrace g \in \pol(M) \: : \: g|_A=f|_A \rbrace$ be a basic open set of $\tau_N$.  Define $X:= \lbrace \theta(\bar{a}) \: : \: \bar{a} \in A \rbrace$.  Then $U= \lbrace g \in \pol(M) \: \: : g|_X = f|_X \rbrace$.

Therefore $\tau_N=\tau_M^+$.
\end{proof}

\begin{cor}\label{Cor:AH}
If $\pol(M)$ is a locally moving clone then $\pol(M)$ has automatic homeomorphicity with respect to all polymorphism clones.
\end{cor}
\begin{proof}

Let $\theta: \pol(M) \rightarrow \pol(N)$ be an isomorphism.  Proposition \ref{thm:cloneinduced} shows that $\theta$ is a homeomorphism between $\pol(N)$ and $\pol(M)$ with respect to the topology from the actions on the Boolean algebras.  We may apply Theorem \ref{theorem:LMtopisRight} to $\pol(M)$ to find that the topology on $\pol(M)$ from its action on $M$ is the same as the topology from its action on $B(\pol(M))$.  Lemma \ref{Remark:FODef} shows that $\pol(N)$ is also locally moving and so the topology on $\pol(N)$ from its action on $N$ is the same as the topology from its action on $B(\pol(N))$.  Therefore $\theta$ is a homeomorphism with respect to the desired topologies.

\end{proof}

We will talk about the properties of the structures we have built here, using the examples of locally moving clones we will find in the rest of the paper, so I would suggest skipping this part until you have read the rest of the paper.  Let $M$ be a first order model with locally moving $\pol(M)$.  Recall the map $\theta$ from the proof of Theorem \ref{theorem:LMtopisRight}.

\[\theta(\varBA(\alpha)) := \lbrace x \in M \: : \:\forall \gamma \in \aut(M) \: \gamma(x) \not= x \Rightarrow  \varBA(\gamma) \cap \varBA(\alpha) \neq \emptyset \rbrace\]

This map is injective, maps $\lbrace 0,1 \rbrace$ to $\lbrace \emptyset, M \rbrace$ and either preserves or inverts $\subseteq$.  Roughly speaking, $\theta(\varBA(\alpha))$ corresponds to the closure of the support of $\alpha$.  If $M = (\mathbb{Q},<)$ then $\theta(\varBA(\alpha))$ is the set of all $q \in \mathbb{Q}$ such that if $\gamma(q) \not= q$ then $\supp(\alpha) \cap \supp(\gamma) \not= \emptyset$, so $\theta(\varBA(\alpha))$ is $\overline{\supp(\alpha)}$.  In almost all circumstances, $\theta$ will \emph{not} be an embedding of Boolean algebras.  We may need to reverse the order, so the roles of union and intersection may interchange.  Additionally, if $f \in \pol(M)$ then we only have $f(\theta(\varBA(\alpha)))\subseteq \theta(f(\varBA(\alpha)))$.  For example, if $f \in \pol(\mathbb{Q},<)$ is such that $\mathrm{im}(f)= \lbrace \frac{m}{2^n} \: : \: m,n \in \mathbb{Z} \rbrace$ and $\alpha \in \aut(\mathbb{Q},<)$ is such that $\supp(\alpha)=\mathbb{Q}$ then
\[f(\theta(\varBA(\alpha)))= \left\lbrace \frac{m}{2^n} \: : \: m,n \in \mathbb{Z} \right\rbrace\subsetneq \mathbb{Q} = \theta(f(\varBA(\alpha)))\]
However, it is true that $\overline{f(\theta(\varBA(\alpha)))} = \theta(f(\varBA(\alpha)))$.

In the proof of Theorem \ref{theorem:LMtopisRight}, the structure $N$ is the space of all ultrafilters of $B(\pol(M))$, and $M^+$ is the set of filters of $\mathcal{P}(M)$.  If $M=\mathbb{Q}$ then both $N$ and$ M^+$ are closer to $\mathbb{R}$ than $\mathbb{Q}$.  Only in very nice circumstances will we have $N=M^+$.  For example, let $M:= \mathbb{Q} \times \lbrace 1,2 \rbrace $ be the strucutre with order $ (p,i) < (q,j)$ if and only if $p<q$ or $p=q$ and $i<j$ (i.e. the linear order obtained by replacing every element of the rationals by the 2 element linear order).  Then $N \cong \mathbb{R}$, but the ultrafilters of $M^+$ are isomorphic to $\mathbb{R} \times \lbrace 1,2 \rbrace$.  In this case, each $n \in N$ occurs in $\mathcal{P}(M^+)$ as the union of a rigid substructure, i.e. $\lbrace (q,i) \: : \: i =1,2 \rbrace$ for some $q \in \mathbb{R}$.

One consequence of these observations is that if we know that $\Delta$ is locally moving then we know that $\Delta$ is a subclone of $\pol(M)$ for some $M$, and that the complete atomless Boolean algebra that witnesses that $\Delta$ is locally moving exists as a suborder of $\mathcal{P}(M)$.  Thankfully, this means that the Group condition does not require us to quantify over all complete atomless Boolean algebras in practice.

\section{Methods for Proving a Clone is Locally Moving}

\begin{dfn}
Let $M$ be a topological space and let $X \subseteq M$.
\[
\begin{array}{rcl}
\interior{X} & := & \bigcup \lbrace Y \subseteq X \: : \: Y \textnormal{ is open} \rbrace \\
\overline{X} & := &  \bigcap \lbrace Y \supseteq X \: : \: Y \textnormal{ is closed} \rbrace 
\end{array}
\]
$X$ is called \emph{regular open} if $\interior{\overline{X}}=X$.  Let $X_j \subseteq M$ be regular open sets for $j \in J$.
\[
\begin{array}{rcl}
\bigsqcup\limits_{j \in J} X_j &:=& \interior{\overline{\bigcup X_j}}  \\
\bigsqcap\limits_{j \in J} X_j &:=& \interior{\overline{\bigcap X_j}}
\end{array}
 \]
Additionally, $X_0 \sqcup X_1 := \bigsqcup\limits_{j \in \lbrace 0, 1 \rbrace} X_j $ and $X_0 \sqcap X_1 := \bigsqcap\limits_{j \in \lbrace 0, 1 \rbrace} X_j $.  We also define $- X := \interior{\overline{M \setminus X}}$.
\end{dfn}

\begin{prop}\label{ROpenImpliesCBA}
$(\mathcal{X}, \sqcup, \sqcap, - , \emptyset, M)$ is a complete Boolean algebra.
\end{prop}
\begin{proof}
That the regular open sets of a topology form a complete Boolean algebra is a result that can be found in a number of sources, such as Theorem 1 of Chapter 10 of \cite{GivantBoolean}.
\end{proof}

\begin{dfn}\label{Dfn:Conjugable}
Let $M_1, M_2$ be first order models, and let $f: M_1 \rightarrow M_2$.  Let $\alpha \in \aut(M_1)$ and let $\beta \in \aut(M_2)$.  If $p \in S(\mathrm{im}(f))$ then:
\begin{enumerate}
\item $\beta(p) \in S(\mathrm{im}(f))$ is the type obtained by replacing every parameter $a$ by $\beta(a)$;
\item $\alpha(p) \in S(\mathrm{im}(f))$ is the type obtained by replacing every parameter $a$ by $f \alpha f^{-1} (a)$; and
\item if $M_1 = M_2$ then $f^{-1}(p) \in S(M_1)$ is the type obtained by replacing every parameter $a$ by $f^{-1}(a)$.
\end{enumerate}
If $p \in S(M_1)$ and $M_1=M_2$ then $f(p) \in S(\mathrm{im}(f))$ is the type obtained by replacing every parameter $a$ by $f(a)$.  We say $f$ is \emph{conjugable with respect to } $G \subseteq \aut(M)$ if for all $p \in S(\mathrm{im}(f))$ and all $\beta \in G$
\[ | \lbrace x \in M \: : \: M \models \beta(p)(x) \rbrace | = | \lbrace x \in M \: : \: M \models p(x) \rbrace | \]
\end{dfn}

\begin{remark}
If $f$ is conjugable then for all $\alpha, \beta \in \aut(\mathbb{Q})$ both $ f\alpha$ and $ \beta f$ are too.
\end{remark}

\begin{prop}\label{Prop:LocallyMovingMonoid}
Let $M$ be a first order model, let $G \subseteq \aut(M)$ and let $f \in \emb(M)$.  Suppose that for all $\alpha \in \aut(M)$ there is a $\beta \in G$ such that $\varBA(\alpha)=\varBA(\beta)$.  If $f$ is conjugable w.r.t. $G$ then it is algebraically canonical.
\end{prop}
\begin{proof}
Let $f \in \emb(M)$ be conjugable, and let $\alpha \in \aut(M)$.  Let $\gamma \in G$ be such that $\varBA(\alpha)=\varBA(\gamma)$.  If $x \in \mathrm{im}(f)$ then $\beta(x):=f \alpha f^{-1} (x)$.  At this stage, $\beta$ is a partial automorphism, as $\gamma$ is an automorphism and the isomorphism $f(M) \cong M$ is witnessed by $f$.  We extend $\beta$ to a total automorphism using a back-and-forth procedure.

Let $x \in M \setminus \mathrm{im}(f)$.  Let $I(x):= \lbrace z \in M \: : \: \mathrm{tp}_{\mathrm{im}(f)}(z) = \mathrm{tp}_{\mathrm{im}(f)}(x) \rbrace$.  Since $f$ is conjugable, there is a $y$ such that $\mathrm{tp}_{\mathrm{im}(f)}(y)= \beta(\mathrm{tp}_{\mathrm{im}(f)}(x))$ and, $|I(y)| = |I(x)|$, so we can extend the domain of $\beta$ to include $I(x)$.  Let $y \in M \setminus \mathrm{im}(f)$.  Since $f$ is conjugable, $f \alpha^{-1}$ is also conjugable, and hence there is an $x$ such that $\mathrm{tp}_{\mathrm{im}(f)}(x)= \beta^{-1}(\mathrm{tp}_{\mathrm{im}(f)}(y))$.  Again, $I(y) \cong I(x)$, so we can extend the range of $\beta$ to include $I(y)$.

Therefore if $f$ is conjugable w.r.t. $G$ then it is algebraically canonical.
\end{proof}

\begin{dfn}
Let $P$ be a partial order.  $P$ is said to be \emph{Dedekind-MacNeille complete} if for every $I \subseteq P$ such that $I = \lbrace x \in P \: : \: \forall y \in P ((\forall z \in I \: y \geq z ) \rightarrow x \leq y ) \rbrace$ there is an $a \in P$ such that $I= \lbrace x \in P \: : \: x \leq p \rbrace$.
\end{dfn}

\begin{lemma}
If $P$ is a partial order then there is a minimal partial order $P^+$ such that $P^+$ is Dedekind-MacNeille complete.   If $\alpha \in \aut(P)$ then there is a unique extension of $\alpha$ to $P^+$.
\end{lemma}
\begin{proof}
This is a classical result, which can be found from a number of sources, for example \cite{Warren1997} by Warren.
\end{proof}

\section{The Polymorphism Clones of the Rationals}\label{Section:Q}

Let $(\mathbb{Q},\mathcal{L})$ be any first-order reduct of the rationals.

\begin{theorem}[(Cameron, \cite{CameronReductsOfQ})]
Let $\updownarrow$ be an isomorphism from $(\mathbb{Q},<)$ to $(\mathbb{Q},>)$.  Let $\circlearrowright_1$ be an isomorphism from $(-\infty,\pi)$ to $(\pi,\infty)$, and let $\circlearrowright := \circlearrowright_1 \cup \circlearrowright_1^{-1}$.  Then:
\begin{enumerate}
\item $\aut(\mathbb{Q},\mathcal{L}) = \aut(\mathbb{Q},<)$;
\item $\aut(\mathbb{Q},\mathcal{L}) = \langle \aut(\mathbb{Q},<), \updownarrow \rangle$;
\item $\aut(\mathbb{Q},\mathcal{L}) = \langle \aut(\mathbb{Q},<), \circlearrowright \rangle$;
\item $\aut(\mathbb{Q},\mathcal{L}) = \langle \aut(\mathbb{Q},<), \updownarrow,\circlearrowright \rangle$; or
\item $\aut(\mathbb{Q},\mathcal{L}) = \aut(\mathbb{Q},=)$.
\end{enumerate}
\end{theorem}

\begin{dfn}
Let $\tau$ be the topology on $\mathbb{Q}$ which is generated by basis $ \lbrace (a,b) \: : \: a,b \in \mathbb{Q} \rbrace $.  Let $\mathcal{X}$ be the set of regular open sets of $\mathbb{Q}$.
\end{dfn}

$\mathcal{X}$ is preserved by $\aut(\mathbb{Q},\mathcal{L})$ unless $\aut(\mathbb{Q},\mathcal{L})=\aut(\mathbb{Q},=)$.

\begin{lemma}
If $I$ is an open interval then $I \in \mathcal{X}$.
\end{lemma}
\begin{proof}
Since $I$ is an open interval, there are $\alpha, \beta \in \mathbb{R} \cup \lbrace -\infty, \infty \rbrace$ such that $I=(\alpha,\beta)$.  Then $\overline{I}=[\alpha,\beta]$ and $\interior{\overline{I}}=(\alpha,\beta)$.
\end{proof}

\begin{remark}\label{Lemma:atomless}
If $X \in \mathcal{X}(\mathbb{Q})$ is non-empty then there is an open interval $I$ such that $I \subseteq X$.
\end{remark}

\begin{prop}\label{prop:Qsupport}
For any interval $(a,b)$ for $a,b \in \mathbb{Q}$ there is a $\alpha_{(a,b)} \in \aut(\mathbb{Q},<)$ such that $\supp(\alpha_{(a,b)})$ is $I$.
\end{prop}
\begin{proof}
We define $\alpha_{(a,b)}$ to be the following function:
\[
\alpha_{(a,b)}(x) = \left\lbrace
\begin{array}{l l}
x & x \not\in I \\
a+ \frac{(x-a)^2}{b-a} & x \in I
\end{array}
\right.
\]

\end{proof}

The ingredients for showing that $\aut(\mathbb{Q},\mathcal{L})$ is locally moving are ready.  We must now work on the Monoid Condition.

\begin{dfn}\label{dfn:quadratic}
$\alpha \in \aut(\mathbb{Q}, \mathcal{L})$ is called \emph{quadratic} if there are convex pairwise disjoint sets $I_n$, such that if $x \in I_n$ then $\alpha(x)= q_{0,n} + q_{1,n} x + q_{2,n} x^2$ where $q_{i,n} \in \mathbb{Q}$ and if $x \not\in \bigcup I_n$ then $\alpha(x)=x$.

Let $G$ be the set of all quadratic automorphisms.
\end{dfn}

\begin{lemma}\label{lemma:GforQ}
For all $\alpha \in \aut(\mathbb{Q},\mathcal{L})$ there is a $\beta \in G$ such that $\varBA(\alpha)=\varBA(\beta)$.
\end{lemma}
\begin{proof}
Let $\alpha \in \aut(\mathbb{Q},\mathcal{L})$, and let $(a_i,b_i)$ be the maximal convex sets of $\supp(\alpha)$.  Let $\alpha_{(a_i,b_i)}$ be the automorphisms from Proposition \ref{prop:Qsupport}.  Then $\alpha_{(a_i,b_i)} \in G$.  We define $\beta$ as follows:

\[
\begin{array}{rcl}
\beta &:=& \mathrm{id}|_{\mathbb{Q}\setminus \supp(\alpha)} \cup \bigcup_{i} \alpha_{(a_i,b_i)} |_{(a_i,b_i)}
\end{array}
\]

Then $\beta \in G$ and $\supp(\beta) = \supp(\alpha)$, so $\varBA(\alpha)=\varBA(\beta)$.
\end{proof}

\begin{lemma}\label{Lemma:LocallyMovingMonoid}
Let $f \in \emb(\mathbb{Q}, \mathcal{L})$.  There is a $g \in \emb(\mathbb{Q}, \mathcal{L})$ such that $g$ and $gf$ are conjugable with respect to $G$.
\end{lemma}
\begin{proof} To help keep track of all the different copies of $\mathbb{Q}$ in this proof, let $\mathbb{Q}_1, \mathbb{Q}_2 \cong \mathbb{Q}$ be such that $f: \mathbb{Q}_1 \rightarrow \mathbb{Q}_2$.

We define $\mathbb{R}^+_i$ to be:
\begin{itemize}
\item $\mathbb{R} \cup \lbrace - \infty, \infty \rbrace $ if $\circlearrowright \not\in \aut(\mathbb{Q},\mathcal{L})$
\item $\mathbb{R} \cup \lbrace \pm \infty \rbrace $ if $\circlearrowright \in \aut(\mathbb{Q},\mathcal{L})$
\end{itemize} 

Let $G_i$ be the copy of $G$ in $\aut(\mathbb{Q}_i,\mathcal{L})$.  Let $\alpha \in G_1$.  Note that $\alpha$ has a unique extension to $\aut(\mathbb{R}_1^+, \mathcal{L})$, which we also denote by $\alpha$.

Let $x \in \mathbb{Q}_2 \setminus \mathrm{im}(f)$.  There is an $a \in \mathbb{R}_1^+$ such that $f(\mathrm{tp}_{\mathbb{Q}}(a))=\mathrm{tp}_{\mathrm{im(f)}}(x)$.  For each $\gamma \in G_2$, the $\mathrm{im}(f)$-types $f(\mathrm{tp}_{\mathrm{im}(f)}(\gamma(a)))$ are realised in $\mathbb{R}_2^+$.  Let $I(x, \gamma) \subseteq \mathbb{R}^+_2$ be the sets of realisations of $f(\mathrm{tp}_{\mathbb{Q}_1}(\gamma(a)))$.  If $I(x,\gamma)$ is finite let $J(x,\gamma):=I(x,\gamma)$, but if $I(x,\gamma)$ is infinite then $J(x,\gamma):=I(x,\gamma)\cap \mathbb{Q}_2$.  There are countably many such $J(x,\gamma)$, and each $J(x,\gamma)$ is finite or countable. 

We now define
\[\mathbb{Q}_3:= (\mathbb{Q}_2 \times\mathbb{Q}) \cup \lbrace I(x,\gamma)\times \mathbb{Q} \: : \: x \in \mathbb{Q}_2 \setminus \mathrm{im}(f) \quad \gamma \in G_2 \rbrace\]

We define $g: \mathbb{Q}_2 \rightarrow \mathbb{Q}_3$ to be $g(x):= (x,0)$.  We allowed $\gamma$ to range over $G_2$ in the construction of $\mathbb{Q}_3$, so if $p \in S(\mathrm{im}(g))$ is realised then $\gamma(p)$ is also realised.  By using $I(x,\gamma) \times \mathbb{Q}$, we guarantee that for all $p \in S(\mathrm{im}(g))$ and all $\gamma \in \aut(\mathbb{Q}_3,\mathcal{L})$
\[ |\lbrace x \in \mathbb{Q} \: : \: \mathrm{tp}_{\mathrm{im}(g)}(x)=p \rbrace | = \aleph_0\]
thereby showing that $g$ is conjugable with respect to $G$.  The argument showing that $gf$ is conjugable is virtually identical.
\end{proof}

\begin{theorem}
If $\pol(\mathbb{Q},\mathcal{L})$ is such that:
\begin{enumerate}
 \item $\aut(\mathbb{Q},\mathcal{L}) \not= \aut(\mathbb{Q},=)$; and
 \item  $\mon(\mathbb{Q},\mathcal{L})= \emb(\mathbb{Q},\mathcal{L})$,
\end{enumerate}
then $\pol(\mathbb{Q},\mathcal{L})$ has automatic homeomorphicity.
\end{theorem}
\begin{proof}
Since $\aut(\mathbb{Q},\mathcal{L}) \not= \aut(\mathbb{Q},=)$, it preserves $\mathcal{X}(\mathbb{Q})$, and hence acts on it as a complete atomless Boolean algebra.  $\aut(\mathbb{Q},<) \leq \aut(\mathbb{Q},\mathcal{L})$ for all $\mathcal{L}$, so for all intervals $I$ there is an $f \in \aut(\mathbb{Q},\mathcal{L})$ such that $\varBA(f)=\supp(f)=I$.  For every $X \in \mathcal{X}$ there is an interval $I$ such that $I \subseteq X$, so $\aut(\mathbb{Q}, \mathcal{L})$ is a locally moving group, and $\pol(\mathbb{Q},\mathcal{L})$ satisfies the group condition.

Lemmas \ref{Prop:LocallyMovingMonoid} and \ref{Lemma:LocallyMovingMonoid} show that $\pol(\mathbb{Q},\mathcal{L})$ satisfies the monoid condition, and hence is a locally moving clone. Corollary \ref{Cor:AH} shows that $\pol(\mathbb{Q},\mathcal{L})$ has automatic homeomorphicity.
\end{proof}

\section{The Homogeneous Binary Branching C-relation on Leaves}

The setting for this section is again an ordering, this time the semilinear order known as $\mathbb{S}_2$. We use $x \parallel y$ to indicate that $x$ and $y$ are incomparable.  $\mathbb{S}_2$ has a full description in \cite{GenericTree}, but for our purposes, we just need to know the following things: 

\begin{enumerate}
\item For all $x,y,z$, if $x,y \leq z$ then $x \leq y$ or $y \leq x$.
\item For all $x,y$ there is a $z$ such that $z \leq x,y$.
\item $\mathbb{S}_2$ is binary branching, i.e. if $x_1,x_2,x_3$ are pairwise incomparable then there is a $y$ such that $y \parallel x_3$ but $y \leq x_1,x_2$, up to permuting the indices.
\item $\mathbb{S}_2$ is ultrahomogeneous in the language $\lbrace \leq, C \rbrace$, where $C(z;x,y)$ is the relation
\[C(z;x,y) \Leftrightarrow x \parallel y \wedge \exists u ( x< u \wedge y<u \wedge u \parallel z )\]
\end{enumerate}

\begin{dfn}
Let $\mathbb{L}_2$ be a maximal pairwise incomparable subset of $\mathbb{S}_2$.  The structure $(\mathbb{L}_2, C)$ is called the \emph{homogenous binary branching C-relation on leaves}.
\end{dfn}

$(\mathbb{L}_2,C)$ is the countable model of a countably categorical theory.  It is homogeneous and universal for finite $C$-relations.  The reducts of $\mathbb{L}_2$ were classified in \cite{CRelOnLeaves} by Bodirsky, Jonnson and van Pham.  There are two proper reducts of $(\mathbb{L}_2,C)$ up to first-order interdefinability, $(\mathbb{L}_2,=)$ and $(\mathbb{L}_2,D)$, where
\[D(x,y,u,v) \Leftrightarrow (C(u;x,y) \wedge C(v;x,y)) \vee (C(x;u,v) \wedge C(y;u,v)) \]

\begin{dfn}
Let $\mathbb{T}_2 \subseteq \mathbb{S}_2$ be the following set:
\[\lbrace x \in \mathbb{S}_2 \: : \: \exists y \in \mathbb{L}_2 \: y \leq x \rbrace \]
\end{dfn}

If $\alpha \in \aut(\mathbb{L}_2)$ then there is a unique extension of $\alpha$ to $\mathbb{T}_2$.

\begin{dfn}
Let $\tau$ be the topology on $\mathbb{L}_2$ obtained by taking the following set as a basis:
\[\lbrace \mathbb{L}_2 \cap (\mathbb{T}_2)^{<x} \: : \: x \in \mathbb{T}_2 \rbrace \cup \lbrace \mathbb{L}_2 \setminus (\mathbb{T}_2)^{<x} \: : \: x \in \mathbb{T}_2  \rbrace\]
Let $\mathcal{Y}$ be the collection of regular open sets.
\end{dfn}

Both $\aut(\mathbb{L}_2,C)$ and $\aut(\mathbb{L}_2,D)$ preserve $\tau$.

\begin{lemma}\label{Lemma:S2support}
If $V \in \lbrace \mathbb{L}_2 \cap (\mathbb{T}_2)^{<x} \: : \: x \in \mathbb{T}_2 \rbrace$ there is an $\alpha \in \aut(\mathbb{L}_2,C)$ such that $\supp(\alpha)= V$.
\end{lemma}
\begin{proof}
Each $(\mathbb{T}_2)^{< x}$ is isomorphic to $\mathbb{T}_2$, so let $\mu: \mathbb{T}_2 \rightarrow (\mathbb{T}_2)^{< x}$.  Let $\alpha' \in \aut(\mathbb{T}_2,\leq)$ have no fixed points.

\[\alpha: x \mapsto \left\lbrace
\begin{array}{c l}
x & x \not\in (\mathbb{T}_2)^{<x} \\
\mu \alpha' \mu^{-1} (x) & x \in (\mathbb{T}_2)^{<x}
\end{array}
\right.\]
has the desired property.
\end{proof}

\begin{lemma}\label{Lemma:S2atomless}
$(\mathcal{Y}, \sqcup, \sqcap, - , \emptyset, M)$ is a complete atomless Boolean algebra.
\end{lemma}
\begin{proof}
Proposition \ref{ROpenImpliesCBA} shows that $\mathcal{Y}$ is a complete Boolean algebra.  All basic open sets of $\tau$ are clopen, so are regular open.  Let $V$ be a basic open set contained in regular open $U$.  Suppose $V=\mathbb{L}_2 \cap (\mathbb{T}_2)^{<x}$.  Then there is a $y \in \mathbb{S}_2$ such that $y<x$.  Then $\mathbb{L}_2 \cap (\mathbb{T}_2^{<x}) \subsetneq U$.  Suppose $V = \mathbb{L}_2 \setminus (\mathbb{T}_2)^{<x}$.  Then there is a $y \in \mathbb{T}_2$ such that $y \parallel x$.  Then $\mathbb{L}_2 \cap (\mathbb{T}_2^{<y}) \subsetneq U$.

Therefore $\tau$ is atomless.
\end{proof}

\begin{prop}\label{prop:L1monoid}
Let $f \in \emb(\mathbb{L}_2,\mathcal{L})$.  Then there is a $g \in \emb(\mathbb{L}_2,\mathcal{L})$ such that $g$ and $gf$ are conjugable with respect to $\aut(\mathbb{L}_2, \mathcal{L})$.
\end{prop}
\begin{proof}
As with the rationals, we will label all the different copies of $\mathbb{L}_2$ with numbers, such as $\mathbb{L}_2(i)$.  We let $\mathbb{T}_2(i)$ be the corresponding copy of $\mathbb{T}_2$.  Let $\mathbb{T}_2(i)^+$ be the minimal Dedekind-MacNeille complete extension of $\mathbb{T}_2(i)$.

Let $f: \mathbb{L}_2(1) \rightarrow \mathbb{L}_2(2)$ be a self-embedding.  Let $ x \in \mathbb{L}_2(2) \setminus \mathrm{im}(f)$.  If there are $u,v \in \mathrm{im}(f)$ such that $\mathbb{L}_2(2) \models C(u;x,v)$ then there is some maximal $y \in \mathbb{T}_2(1)^+$ such that $f(y) \geq x$.  Let $y(x)$ be this maximal element.  If $\mathbb{L}_2(2) \models C(x;u,v)$ for all $u,v \in \mathrm{im}(f)$ then let $y(x) := \mathrm{inf}\lbrace z \in \mathbb{T}_2(2) \: : \: \forall u \in \mathrm{im}(f) \:  z \geq x, u \rbrace$.

Let $\alpha \in \aut(\mathbb{L}_2,C)$, let $ t \in \mathbb{T}_2$, and let $ l \in \mathbb{L}_2$ be below $t$.  If $\beta \in \aut(\mathbb{L}_2)$ is such that $\beta(l) < \alpha(t)$ then there is a $\gamma \in \aut(\mathbb{L}_2)$ such that $\gamma \beta(t)=\alpha(t)$ and $\gamma(l)=l$.  Since there are only countably many possible destinations for $l$, and $\gamma|_{(\mathbb{S}_2)^{>l}} \in \aut(\mathbb{Q})$, this means that the $\aut(\mathbb{L},\mathcal{L})$-orbit of $y(x)$ is countable.

\[\mathbb{L}_2(3) := \mathbb{L}_2 \times (\mathbb{L}_2(1) \cup \lbrace \alpha(y(x))\: : \: \alpha \in \aut(\mathbb{L}_2(1)) \textnormal{ and } x \in \mathbb{L}_2(2) \setminus \mathrm{im}(f) \rbrace)\]

We define $C$ on $\mathbb{L}_2(3)$ as follows:

\[C((z,i);(y_1,j_1),(y_2,j_2)) \Leftrightarrow \left\lbrace
\begin{array}{c c}

 i \not= j_1 = j_2 & \mathrm{or} \\
  i=j_1=j_2 \textnormal{ and } C(z;y_1,y_2) & \mathrm{or} \\
 C(i;j_1,j_2) \\

\end{array}
\right.\]

Since $(\mathbb{L}_2, \mathcal{L})$ is a reduct of $(\mathbb{L}_2,C)$, each symbol in $\mathcal{L}$ has a first order definition using $C$.  If we mimic these definitions on $(\mathbb{L}_2(3),C)$, we obtain $(\mathbb{L}_2(3),\mathcal{L}) \cong (\mathbb{L}_2,\mathcal{L})$.

Let $x \in \mathbb{L}_2(2) \setminus \mathrm{im}(f)$.  The set $\lbrace x' \in \mathbb{L}_2(2) \: : \: \mathrm{tp}_{\mathrm{im}(f)}(x')= \mathrm{tp}_{\mathrm{im}(f)}(x) \rbrace$ is a substructure of $\mathbb{L}_2(2)$, and hence embeds into $  \lbrace y(x) \rbrace\times\mathbb{L}_2$.  Therefore $\mathbb{L}_2(2) \subseteq \mathbb{L}_2(3)$.  We take $g: \mathbb{L}_2(2) \rightarrow \mathbb{L}_2(3)$ to be the embedding $g(x)=(x,x)$.  For every $x \in \mathbb{L}_2(3) \setminus \mathrm{im}(g)$
\[|\lbrace x' \in \mathbb{L}_2(3) \: : \: \mathrm{tp}_{\mathrm{im}(g)}(x') = \mathrm{tp}_{\mathrm{im}(g)}(x) \rbrace| = \aleph_0\]
and hence $g$ is conjugable with respect to $\aut(\mathbb{L}_2,\mathcal{L})$.  The argument that $gf$ is conjugable is very similar.

\end{proof}

\begin{theorem}
Let $(\mathbb{L}_2,\mathcal{L})$ be a reduct of $(\mathbb{L},C)$ such that
\begin{enumerate}
\item $\aut(\mathbb{L}_2,\mathcal{L}) \not= \aut(\mathbb{L}_2,=)$; and
\item $\mon(\mathbb{L}_2,\mathcal{L}) = \emb(\mathbb{L}_2,\mathcal{L})$.
\end{enumerate}
Then $\pol(\mathbb{L}_2,\mathcal{L})$ has automatic homeomorphicity.
\end{theorem}
\begin{proof}
Both $\aut(\mathbb{L}_2,C)$ and $\aut(\mathbb{L}_2,D)$ preserve $\tau$, so Lemmas \ref{Lemma:S2support} and \ref{Lemma:S2atomless} together show that $\aut(\mathbb{L}_2,\mathcal{L})$ is a locally moving group.  Proposition \ref{prop:L1monoid} shows that $\mon(\mathbb{L}_2,\mathcal{L})$ satisfies the Monoid Condition.  Therefore $\pol(\mathbb{L}_2,\mathcal{L})$ is locally moving, and hence has automatic homeomorphicity.
\end{proof}

\section{The Countable Atomless Boolean Algebra}

The class of all finite Boolean algebras is a Fra\"{i}ss\'{e} class, with Fra\"{i}ss\'{e} limit $\mathcal{B}$.  The language of Boolean algebras used is $\mathcal{L}_B:= \lbrace 0, 1, \cup, \cap, ^c ,\neq \rbrace$, representing the empty set, the full set, union, intersection, complement and inequality respectively.  As with all Fra\"{i}ss\'{e} limits, it is universal, $\aleph_0$-categorical and ultrahomogeneous.  Its ultrahomogeneity implies that it is atomless, and so its completion is also atomless.  The reducts of $\mathcal{B}$ are not fully classified, but the ones with an entirely functional signature are known, thanks to  Bodor, Kalina and Szab\'{o} \cite{FunctionalReductsBA}.

The most natural candidate for a complete atomless Boolean algebra for $\pol(\mathcal{B})$ to act on is of course $\bar{\mathcal{B}}$, the completion of $\mathcal{B}$.

\begin{lemma}\label{BA:groupcondition}
For every $b \in \bar{\mathcal{B}}$ there is an $a \in \mathcal{B}$ and $\phi \in \aut(\mathcal{B})$ such that $a \leq b$ and if $\phi(x) \neq x$ then $x < a$.
\end{lemma}
\begin{proof}
Rather than show this directly, we will show that $\mathcal{X}$, the algebra of regular open subsets of the rationals discussed in Section \ref{Section:Q}, is isomorphic to $\bar{B}$.  Note that $\mathcal{X}$ is the completion of the Boolean algebra of finite unions of intervals of $\mathbb{Q}$, which we call $\mathfrak{x}$.  This $\mathfrak{x}$ is both countable and atomless, and hence $\mathcal{B} \cong \mathfrak{x}$, and $\bar{\mathcal{B}} \cong \mathcal{X}$.  Since $\aut(\mathbb{Q},<)$ is locally moving, and $\aut(\mathbb{Q},<) \leq \aut(\mathcal{B},\cup,\cap, 0,1, \,^c)$, we have this lemma.
\end{proof}

Thus $\aut(\mathcal{B},\cup,\cap, 0,1, \,^c)$ fulfils The Group Condition.  We turn out attention to The Monoid Condition.

\begin{dfn}
Let $E(\mathcal{B}) $ be the space of finite unions of intervals from $\mathbb{Q} \times \mathbb{Q}$ with the lexicographic order.  Let $e: \mathcal{B} \rightarrow E(\mathcal{B})$ be the function $
e(I) := I \times I
$ where $I \in \mathcal{X}$.
\end{dfn}

\begin{lemma}\label{BA:EverythingIsRepresented}
For every $\alpha \in \aut(\mathcal{B})$ there is a $\beta \in G$ such that $\varBA(\alpha)=\varBA(\beta)$.
\end{lemma}
\begin{proof}
Lemma \ref{lemma:GforQ} shows that for every regular open subset $I \subseteq \mathbb{Q}$ there is a $\beta \in G$ such that $\varBA(\beta)=I$.  Since the algebra of regular open subsets of $\mathbb{Q}$ is isomorphic to $\bar{\mathcal{B}}$, this means that there is a $\beta$ such that $\varBA(\beta)=\varBA(\alpha)$.
\end{proof}

\begin{lemma}
$E(\mathcal{B}) \cong \mathcal{B}$ and $e$ is a self-embedding.
\end{lemma}
\begin{proof}
As we have already seen, $\mathcal{B} \cong \mathcal{X}$.  As both are countable dense linear orders, $\mathbb{Q} \times \mathbb{Q}$ with the lexicographic order is isomorphic to $\mathbb{Q}$.  Therefore $\mathcal{B} \cong E(\mathcal{B})$.

$e$ is obviously injective, so we need only check that it is a homomorphism.  For the constants, $e(\emptyset)= \emptyset \times \mathbb{Q}=\emptyset$ and $e(\mathbb{Q}) = \mathbb{Q} \times \mathbb{Q} = \mathbb{Q}$.  For the operations:
$$
\begin{array}{c c c}
\begin{array}{rcl}
e(I \cup J) & = & (I \cup J )\times \mathbb{Q} \\
& = & (I \times \mathbb{Q} ) \cup (J \times \mathbb{Q}) \\
& = & e(I) \cup e(J) \\

\end{array}
\begin{array}{rcl}

e(I \cap J) & = & (I \cap J )\times \mathbb{Q} \\
& = & (I \times \mathbb{Q} ) \cap (J \times \mathbb{Q}) \\
&=& e(I) \cap e(J)
\end{array}
\begin{array}{rcl}
e(\mathbb{Q} \setminus I) & = & (\mathbb{Q} \setminus I )\times \mathbb{Q} \\

& = & (\mathbb{Q} \times \mathbb{Q} ) \setminus (I \times \mathbb{Q}) \\
& = & \mathbb{Q} \setminus e(I)
\end{array}

\end{array}
$$
\end{proof}

\begin{lemma}\label{BA:conjugable}
For all $f \in \emb(\mathcal{B})$ there is a $g \in \emb(\mathcal{B})$ such that both $g$ and $gf$ are conjugable with respect to $G$, the set of all quadratic automorphisms of $\mathbb{Q}$ (Definition \ref{dfn:quadratic}).
\end{lemma}
\begin{proof}
As usual, we use indices to denote which copy of $\mathcal{B}$ we are considering.  Let $f:\mathcal{B}_1 \rightarrow \mathcal{B}_2$.  Let $\mathcal{B}_3$ be the subset of $\bar{\mathcal{B}}$ that contains the set $G(x):= \lbrace y \in \bar{\mathcal{B}} \: : \: \exists g \in G \: : \: g(x)=y \rbrace$ for all $x \in \mathcal{B}_2 \setminus \mathrm{im}(f)$.  Let $g: \mathcal{B}_2 \rightarrow E(\mathcal{B}_3)$ be $e: \mathcal{B}_3 \rightarrow E(\mathcal{B}_3)$ with the domain restricted to $\mathcal{B}_2$.

Let $\alpha \in G$.  We seek to find a $\beta$ such that $\beta gf= gf \alpha$.  If $x \in \mathrm{im}(gf)$ then we choose $\beta(x):=gf\alpha (gf)^{-1} (x)$, so now we turn our attention to $x \in E(\mathcal{B}_3) \setminus \mathrm{im}(gf)$.  Suppose $e^{-1}(x) \in \mathcal{B}_2$, and let $\tilde{e}:E(\mathcal{B}_3) \rightarrow \mathcal{B}_3$ be the function given by $\tilde{e}((a,b))= a$.  For all $y \in E(\mathcal{B}_3)$, the substructure $\tilde{e}^{-1}(y)$ is isomorphic to $\mathcal{B}$.  The ultahomogeneity of $\mathcal{B}$ shows that there is an isomorphism $phi_x$ from $\tilde{e}^{-1}(x)$ to $\tilde{e}^{-1}(\beta(e(\tilde{e}^{-1}(x))))$ that maps $e(\tilde{e}^{-1}(x))$ to $\beta(e(\tilde{e}^{-1}(x)))$.  For all $z \in \tilde{e}^{-1}(x)$ we define $\beta(z):= \phi_x(z)$.

We finally consider $x \in E(\mathcal{B}_3) \setminus \mathrm{im}(gf) $, with $e^{-1}(x) \not\in \mathcal{B}_3$.  By construction, since $\mathrm{tp}_{\mathrm{im}(f)}(e^{-1}(x))$ is realised and $\alpha \in G$, the type  $\alpha(\mathrm{tp}_{\mathrm{im}(f)}(e^{-1}(x)))$ is also realised.  Let $z$ realise $\alpha(\mathrm{tp}_{\mathrm{im}(f)}(e^{-1}(x)))$.  The substructures $\tilde{e}^{-1}(x)$ and $\tilde{e}^{-1}(z)$ are both isomorphic to $\mathcal{B}$, so there is an isomorphism $\phi_x:\tilde{e}^{-1}(x) \rightarrow \tilde{e}^{-1}(z)$, and if $\tilde{e}^{-1}(y)=\tilde{e}^{-1}(x)$ then $\beta(y):=\phi_x(y)$.

The case for $g$ is very similar.
\end{proof}

\begin{theorem}
$\pol(\mathcal{B},\mathcal{L}_B)$ has automatic homeomorphicity.
\end{theorem}
\begin{proof}
Lemma \ref{BA:groupcondition} shows that $\pol(\mathcal{B}, \mathcal{L}_B)$ satisfies the Group Condition.  Lemmas \ref{BA:EverythingIsRepresented} and \ref{BA:conjugable} combine with Proposition \ref{Prop:LocallyMovingMonoid} to show that $\pol(\mathcal{B}, \mathcal{L}_B)$ satisfies the Monoid Condition.  Therefore $\pol(\mathcal{B}, \mathcal{L}_B)$ is locally moving, and has automatic homeomorphicity.
\end{proof}

\bibliography{total}
\bibliographystyle{plain}

\end{document}